\documentclass{amsart}
\usepackage{amssymb,amsfonts,latexsym}
\newtheorem{thm}{Theorem}[section]
\newtheorem{lem}[thm]{Lemma}
\newtheorem{condition}[thm]{Condition}
\newtheorem{prop}[thm]{Proposition}
\newtheorem{cor}[thm]{Corollary}

\theoremstyle{definition}
\newtheorem{res}[thm]{Result}
\newtheorem{example}[thm]{Example}

\def\zl{Z(U(L))}
\def\zh{Z(U(H))}

\def\uhl{U(H)_{\lambda}}
\def\adl{adL}
\def\sl{Sz(U(L))}
\def\sh{Sz(U(H))}

\begin{document}
\title[Stability of weight spaces]{On the stability of weight spaces of enveloping algebra in prime characteristic}

\date{\today}

\maketitle
\centerline {\small Gil Vernik}
\centerline {\small Mathematisches Seminar, }
\centerline {\small Christian-Albrechts-Universität zu Kiel,}
\centerline {\small Ludewig-Meyn Str. 4, 24098 Kiel,}
\centerline {\small Germany.}
\centerline {\small \tt vernik@math.uni-kiel.de}

\begin{abstract}
By the result of Dixmier, any weight space of enveloping algebra of Lie algebra $L$ over a field of characteristic 0 is $adL$ stable. In this paper we will show that this result need not be true, if $F$ is replaced by a field of prime characteristic. A condition will be given, so a weight space will be $adL$ stable.
\end{abstract}

\section{Introduction}
Let $L$ be a finite dimensional Lie algebra over a field $F$, $U(L)$ its enveloping algebra with center $\zl$. For each linear form $\lambda : L \longrightarrow  F$ we denote 
\[
U(L)_{\lambda} = \{u \in U(L)\ |\ [x,u] = \lambda(x)u\ for\ all\ x \in L \}.
\]
Clearly $U(L)_{\lambda}$ is a submodule of $U(L)$ and $U(L)_{\lambda} U(L)_{\mu} \subseteq U(L)_{\lambda+\mu}$. We call $U(L)_{\lambda}$ a weight space of $U(L)$ and the sum of all $U(L)_{\lambda}$ is direct and is denoted by $\sl$, the semi-center of $U(L)$. Semi-center was introduced by Dixmier \cite[4.3]{dixmBook} and its crucial properties in characteristic 0 is that any two sided ideal of $U(L)$ has non trivial intersection with semi-center ( \cite[4.4.1]{dixmBook} ). If $L$ is nilpotent or $[L,L]=L$ then $\zl = \sl$. In addition if $L$ is nilpotent and $charF=0$, then by \cite{dixmBook} $\sl$ is factorial domain and if $charF=p>0$ then $\sl$ is again factorial by \cite{B}. If $charF=0$ and $L$ is solvable, then $\sl$ is factorial due to Moeglin \cite{M}. However, if char$F=p>0$, then by recent result  \cite{braunvernik}, it is shown that $\sl$ need not be 
factorial. \newline \\
Consider $L=Fx+H$  a Lie algebra over an algebraically closed field $F$ where $H$ is an ideal in $L$. A weight space $U(H)_{\lambda}$ is called $adL$ stable if $[L,U(H)_{\lambda}] \subseteq U(H)_{\lambda}$. By Dixmier \cite[1.3.11]{dixmBook}, when $charF=0$, $U(H)_{\lambda}$ is always $adL$ stable which implies that $\sh$ is also $adL$ stable. The techniques used by Dixmier are characteristic zero dependent and so we cannot adapt them for the case of $charF=p>0$. In this paper we focus on the problem related stability of $\uhl$ in case $L$ is a finite dimensional Lie algebra over a field $F$ of char$F=p>0$. Unlike in the $charF=0$ case, it appears that $\uhl$ is not necessarily $adL$ stable. We will provide conditions, so that $\uhl$ will be $adL$ stable, as well we will show a surprising result that $\sh$ is $\adl$ stable if and only if $\uhl$ is $\adl$ stable for every weight $\lambda$. 
\section{Main result}
Let $L=Fx+H$ be a finite dimensional Lie algebra over a field $F$ and $H$ an ideal in $L$. As was mentioned, by \cite[1.3.11]{dixmBook} if $charF=0$ then $\uhl$ is $adL$ stable. The following discussion will show that this result need not be true if $charF=p>0$. 
For the simplicity we assume $p=3$, although the techniques we use, can be adapted for any $p>0$. 
\begin{res}\label{fres}
Let $L$ be 5-dimensional Lie algebra over a field $F$ of char$F=3$, with the following multiplication table
\[
	[x,e_{1}]=0,\ [x,e_{2}]=e_{2},\ [x,e_{3}]=2e_{3}, [y,e_{1}]=e_{3},\ [y,e_{2}]=e_{1},\ [y,e_{3}]=e_{2},
\]
\[
	[x,y]=2y,\ the\ rest\ of\ products\ 0.
\]
The subspace $K \equiv span\{e_{1},e_{2},e_{3}\}$ is an ideal of $L$ and consider $L=Fx+H$, $H=Fy + K$. Clearly $H$ is a codimension one ideal in $L$. 
Then by \cite[Section 9]{braunvernik}
\begin{equation}\label{first}
\sl= U(K)^{ady}[x^{p}-x,y^{p^{2}}-y^{p}],\ Z(U(H)) = U(K)^{ady}[y^{p^{2}}-y^{p}].
\end{equation}
Therefore $\sl = Z(U(H))[x^{p}-x]$ and using direct calculations we can show that
\[
Z(U(H))) = F[y^{p^{2}}-y^{p},e_{1}^{p},e_{2}^{p},e_{3}^{p},e_{1}e_{2}^{2}+e_{1}^{2}e_{3}+e_{2}e_{3}^{2}].
\]
Consider $u \equiv e_{1}+e_{2}+e_{3}$ and $v \equiv e_{2}^{2}+e_{1}e_{3}+2e_{2}e_{3}+2e_{3}^{2}$. Now $[y,u]=u$ 
and $[y,v]=2e_{1}e_{2}+e_{3}^{2}+e_{1}e_{2}+2e_{1}e_{3}+2e_{2}^{2}+e_{2}e_{3}=
2e_{2}^{2}+2e_{1}e_{3}+e_{2}e_{3}+e_{3}^{2} = 2v$. Hence $u \in U(H)_{1}$ ( the weight space of $U(H)$ related to eigen value $1$ ) and $v  \in U(H)_{2}$. It can be seen that no more semi-invariant for $U(H)$ exists and $\sh=Z(U(H))[u,v]$. But then $[x,u]=e_{2}+2e_{3}$ and this implies that $u$ is not semi-invariant of $U(L)$ and consequently $U(H)_{1}$ is not $adL$ stable. In a similar way, we see that $U(H)_{2}$ is not $adL$ stable, since $[x,v] \notin U(H)_{2}$.
\end{res}
We would like to notice that although both $u$ and $v$ are not $L$-semi-invariant, $uv = w$ where $w \equiv e_{1}e_{2}^{2}+e_{1}^{2}e_{3}+e_{2}e_{3}^{2} + e_{2}^{3}+e_{3}^{3}$ and it can be seen that $w \in Z(U(H))$ or by \eqref{first}, $w \in \sl$. We recall the following result of Moeglin \cite[Lemma 2]{M}. Let $L$ be solvable Lie algebra over a field of characteristic 0 and let $u,v$ be non zero elements of $U(L)$, then if $uv$ is semi-invariant, then so are $u$ and $v$. In addition if $F$ is algebraically closed, then $uv \in \sl$ implies that $u,v \in \sl$. This result was also extended in \cite[Proposition 1.3]{oomsnau} and it seems to be crucial in showing that $\sl$ is factorial in characteristic 0. Similar result appears to be wrong if $charF=p>0$. Clearly if $L$ is Lie algebra over a field $F$ of $charF=p>0$ defined by $[x,y]=y$, $[x,z]=-z$, $[y,z]=t$ and $[t,L]=0$, then $yy^{p-1}=y^{p} \in U(L)_{0}$ but both $y, y^{p-1} \notin U(L)_{0}$. However, we are interested in the following condition 
\begin{condition}\label{cond}
Let $L$ be solvable finite dimensional Lie algebra over an algebraically closed field $F$ of char$F=p>0$. Let $u,v \in U(L)$ be non zero such that $u$ and $v$ are linearly independent over $U(L)$. If $uv$ is semi-invariant, then so are $u$ and $v$.
\end{condition}
Result \eqref{fres} is the evidence where this condition fails and by \cite{braunvernik} $\sl$ in this case is not factorial. Based on many observations, we have a strong assumption to believe there is linkage between Condition \eqref{cond} and factoriality of $\sl$ where $L$ is a finite dimensional Lie algebra over algebraically closed field $F$ of $charF=p>0$, however we are not able not supply any general proof here. \newline \\
We would like to mention that being $adL$ stable, does not imply that $\uhl$ is a weight space of $U(L)$, as can be seen in the following example 
\begin{example}
Let $L$ be 4 dimensional Lie algebra spanned by $\{x,y,u_{1},u_{2}\}$ with a multiplication table 
$[x,u_{2}]=u_{1}$, $[y,u_{1}]=u_{1}$, $[y,u_{2}]=u_{2}$ and the rest of the products are 0. 
Then $L=Fx+H$ where $H=span\{y,u_{1},u_{2}\}$. Then $U(H)_{\lambda}=span\{u_{1},u_{2}\}$. Clearly $\lambda=1$ and $U(H)_{1}$ is $adx$ stable. Notice that 
$adx(u_{2})=u_{1}$ but $U(H)_{1} \not\subset \sl$.
\end{example}
The following result will establishes a condition on $L$ for $\uhl$ to be $adL$ stable.
\begin{lem}\label{AR:luma}
Let $L=Fx+H$ be a finite dimensional Lie algebra over an algebraically closed field $F$, of characteristic $p>0$ and $H$ an ideal 
in $L$. If $[L,L]$ is nilpotent , then $U(H)_{\lambda}$ is $adx$ stable and consequently is $adL$ stable.
\end{lem}
\begin{proof}
Let $0 \neq u \in U(H)_{\lambda}$. We need to show that for every $y \in H$, $[y,[x,u]]=\lambda(y)[x,u]$. Indeed
\[
[y,[x,u]]=[[y,x],u] + [x,[y,u]] = [[y,x],u] + \lambda(y)[x,u].
\]
Consider $[[y,x],u]$. Since $y \in H$ then $[y,x] \in H$. Hence 
\[
[[y,x],u] = \lambda([y,x])u
\]
and we need to show that $\lambda([y,x])=0$. Now $[y,x] \in [L,L]$  and so $[y,x]$ is nilpotent, therefore $ad[y,x]$ is nilpotent. Let $k$ be its nilpotency index. Therefore 
\[
0=(ad[y,x])^{k}(u)=\lambda([y,x])^{k}u
\]
hence $\lambda([y,x])^{k}=0$ which implies that $\lambda([y,x])=0$.
 \end{proof}
The previous lemma is somewhat misleading. Unlike in the char$F=0$, case, $U(H)_{\lambda}$ is not necessarily an $L$-module and consequently $\lambda([x,y])$, $x,y \in L$, may be non-zero. The assumption $[L,L]$ being nilpotent is therefore essential here.
The next result links between Lemma \eqref{AR:luma} and completely solvable Lie algebra. Recall that a Lie algebra $L$ is called a completely solvable if there is a finite family $\{I_{i}\}$ of ideals in $L$ such that 
\[
L=I_{1} \supset I_{2} \supset \cdots \supset I_{k} \supset I_{k+1}=(0)\ and\  I_{i}=Fe_{i}+I_{i+1}.
\]
\begin{cor}\label{A:dana}
Let $L$ be a solvable finite dimensional Lie algebra over an algebraically closed field $F$, with char$F=p$. Then $[L,L]$ is nilpotent if and only if $L$ is completely solvable.
\end{cor}
\begin{proof}
The proof is similar to one of Dixmier \cite[1.3.12]{dixmBook} where \cite[1.3.11]{dixmBook} which is used in its proof is replaced by Lemma \eqref{AR:luma}.
\end{proof}
The following result extends Lemma \eqref{AR:luma}.
\begin{prop}\label{stab:gil}
Let $L=Fx+H$ be a finite dimensional Lie algebra over an algebraically closed field $F$ of char$F=p>0$ and $H$ a codimension one ideal in $L$. 
Then $\lambda([L,L])=0$ if and only if $\uhl$ is $adx$ stable.
\end{prop}
\begin{proof}
Assume first that $\uhl$ is $adx$ stable for every $\lambda \in H^{*}$. Hence for every $w \in \uhl$ ,$[x,w] \in \uhl$. 
Therefore for every $y \in H$, 
\begin{equation}\label{equ}
[y,[x,w]]=\lambda([y,x])w + \lambda(y)[x,w].
\end{equation}
If $[x,w]=0$, then \eqref{equ} implies that 
\[
0=\lambda([y,x])w + 0
\]
and so $\lambda([y,x])=0$. Assume next that $[x,w] \neq 0$.
If $[y,[x,w]]=0$ then since $[x,w] \in \uhl$ we get $\lambda(y)=0$. By \eqref{equ}, $0=\lambda([y,x])w + 0$ and so  $\lambda[y,x]=0$. 
If $[y,[x,w]] \neq 0$, then the stability of $\uhl$ implies that $\lambda([y,x])=0$.
. Finally  if $[y,x]=0$ then $\lambda([y,x])=0$. Therefore $\lambda([y,x])=0$ for every $y \in H$ and so $\lambda([L,L])=0$. By Lemma \eqref{AR:luma}, the reverse direction is trivial.
\end{proof}
\begin{thm} \label{main:thm}
Let $L=Fx+H$ be a finite dimensional Lie algebra over a field $F$ and $H$ an ideal in $L$. Then $\sh$ is $adx$ stable if and only if $\uhl$ is $adx$ stable for every $\lambda$.
\end{thm}
\begin{proof}
Clearly if $\sh = \zh$, then $\sh$ is $adL$ stable. So we assume that $\zh \subset \sh$. Let $\Delta$ = $\{\mu| \mu$ is weight on $U(H)\}$. Now since $charF=p>0$, then $\sl$ is finitely generated as a $\zl$ module, which implies that $\Delta$ is finite and we assume $dim_{F}\Delta = r$. Assume that $\sh$ is adx stable and assume by negation that there exists some $1 \leq k \leq r$ such such that $U(H)_{\mu_{k}}$ is not $adx$ stable. Let $0 \neq u \in U(H)_{\mu_{k}}$, hence $[x,u] \in \Sigma_{i \leq r}U(H)_{\mu_{i}}$, in particular
\begin{equation}\label{equ_beforeone}
[x,u]= \Sigma_{i \leq r}u_{i},\ u_{i} \in U(H)_{\mu_{i}}.
\end{equation}
Then for every $y \in H$, 
\begin{equation}\label{aqu_one}
[y,[x,u]]=\Sigma_{\mu_{i} \in \Delta} \mu_{i}(y)u_{i}.
\end{equation}
In addition  
\begin{equation}\label{aqu_two}
[y,[x,u]]=[[y,x],u]]+[x,[y,u]]=\mu_{k}([y,x])u + \mu_{k}(y)[x,u].
\end{equation}
Therefore combining \eqref{equ_beforeone}, \eqref{aqu_one} and \eqref{aqu_two} we get 
\begin{equation}\label{equ_sum}
\mu_{k}([y,x])u + \mu_{k}(y)\Sigma_{i \leq r}u_{i} = \Sigma_{\mu_{i} \in \Delta} \mu_{i}(y)u_{i}.
\end{equation}
Since $u_{i}$ belongs to different weight spaces, they are linearly independent, hence the only solution to \eqref{equ_sum} is $\mu_{k}([x,y])=0$ and $\mu_{k}(y) = \mu_{i}(y)$ for every $i \leq r$, which is impossible by the assumption, that $U(H)_{\mu_{k}}$ is not $adx$ stable. Hence $[x,u] \in U(H)_{\mu_{k}}$ and we are done. The second direction is trivial.
\end{proof}
\begin{cor}
Let $L$ be a finite dimensional Lie algebra over an algebraically closed field $F$, of char$F=p>0$ and $H$ any ideal of $L$ such that $[L,L] \subseteq H$. Then $Sz(U(H))$ is $adL$ stable if and only if $\lambda([L,L])=0$ for every $\lambda \in H^{*}$.
\end{cor}
As was seen in \cite[Theorem 4.7]{braunvernik}, if $[L,L]$ is nilpotent then $\sl$ is factorial. By Lemma \eqref{AR:luma} $[L,L]$ being nilpotent, implies that $\uhl$ is $adL$ stable and consequently by Lemma \eqref{stab:gil} $\lambda([L,L])=0$. Therefore $\sl$ is factorial if $\sh$ is $adL$ stable. This make certain relation between factoriality of $\sl$ and stability of semi-center of certain subalgebras. Although the reverse direction is wrong, as $\sl$ is factorial does not implies that $\sh$ is $adL$ stable. We hope this relation will help on further research related factoriality of semi-center. \newline \\

\end{document}